\documentclass[12pt, a4paper,reqno]{amsart}
\usepackage{amsmath}
\usepackage{amssymb,latexsym}
 \usepackage[latin1]{inputenc}
\usepackage[dvips]{graphicx}
\usepackage{color}
\usepackage{txfonts} 
\date{October 15, 2012}
 \usepackage[latin1]{inputenc}
\usepackage{amsthm}

\newtheorem{thm}{Theorem}[section]

\newtheorem{rem}{Remark} [section]
 \newtheorem{prop}{Proposition} [section]
\newtheorem{lemme}{Lemma} [section]
\newtheorem{cor}{Corollary} [section]

\newcommand{\R}{\mathbb R}
\newcommand{\N}{\mathbb N}

\renewcommand{\l}{\lambda}

\begin{document}

\title[Extremal Eigenvalues of the Laplacian]
{Extremal Eigenvalues of the Laplacian on Euclidean domains and closed surfaces  }

\author{Bruno Colbois}
\address{ Universit\'e de Neuch\^atel, Laboratoire de
Math\'ematiques, 13 rue E. Argand, 2007 Neuch\^atel, Switzerland.}
\email{Bruno.Colbois@unine.ch}

\author{Ahmad El Soufi}
\address{Universit\'e de Tours, Laboratoire de Math\'ematiques
et Physique Th\'eorique, UMR-CNRS 7350, Parc de Grandmont, 37200
Tours, France.} \email{elsoufi@univ-tours.fr.}

\thanks{The second author has benefited from the support of the ANR (Agence Nationale de la Recherche) through FOG project ANR-07-BLAN-0251-01.}

\begin{abstract}
We investigate properties of the sequences of extremal values that could be achieved by the eigenvalues of the Laplacian on Euclidean domains of unit volume, under Dirichlet and Neumann boundary conditions, respectively. In a second part, we study sequences of extremal eigenvalues of the Laplace-Beltrami operator on closed surfaces of unit area. 

\end{abstract}


\maketitle

\section{Introduction}

A classical topic in spectral geometry is to investigate upper and lower bounds of eigenvalues of the Laplacian subject to various  boundary conditions and under the fixed volume constraint. Among the most known results in this topic are the  Faber-Krahn inequality for the first Dirichlet eigenvalue, the Szegö-Weinberger inequality for the first positive Neumann eigenvalue on bounded Euclidean domains, and Hersch's inequality for  the first positive eigenvalue on  closed simply connected surfaces. 

\medskip

Just like most of the results one can find in the literature, these sharp inequalities deal with the lowest order positive eigenvalues.  
Aside from numerical approaches, mainly in dimension 2, the determination of  optimal  bounds for  eigenvalues of higher order
 is a problem that remains largely open.

\medskip

In this article our aim will be to show how it is possible, through quite simple considerations, to establish certain intrinsic relationships between the infima (or the suprema) of eigenvalues of different orders. Let us start by fixing some notations. 

\medskip

Given a regular bounded domain $\Omega\subset \R^n$, $n\ge 2$, we designate by $\left\{\lambda_k(\Omega)\right\}_{k\ge 1}$ (resp. $\left\{\mu_k(\Omega)\right\}_{k\ge 0}$) the nondecreasing sequence of eigenvalues of the Laplacian on $\Omega$ with Dirichlet (resp. Neumann) boundary conditions, each repeated according to its multiplicity%(note that we start the numbering of Neumann eigenvalues from zero)
.  We introduce the following universal sequences of real numbers that are attached to the $n$-dimensional Euclidean space : 
$$
{\l}^*_k(n)=\inf \ \{\l_k(\Omega)\ :\ \Omega\subset \R^n,\ |\Omega|=1\}
$$
and
$$
{\mu}^*_k(n)=\sup \ \{\mu_k(\Omega)\ :\ \Omega\subset \R^n,\ |\Omega|=1\},
$$
where $ |\Omega|$ stands for the volume of $\Omega$.
%Notice that the numbering of  the sequence ${\l}^*_k$ starts at $1$ while the numbering of ${\mu}^*_k$ starts at $0$ with ${\mu}^*_0=0$. 
Notice that thanks to standard continuity results for eigenvalues,  the definition of ${\l}^*_k(n)$ (resp. ${\mu}^*_k(n)$) does not change if the infimum (resp. the supremum) is taken only over connected domains.  
The famous Faber-Krahn and Szegö-Weinberger isoperimetric inequalities then read respectively as follows:
$$
{\l}^*_1(n)=\l_1(B^n)\vert B^n\vert^{\frac 2 n} =
j^2_{\frac n2-1,1} \omega_n^{\frac 2 n}  $$
and 
$$
{\mu}^*_1(n)=\mu_1(B^n)\vert B^n\vert^{\frac 2 n} = p^2_{\frac n2,1} \omega_n^{\frac 2 n},
$$
where $\omega_n $ is the volume of the unit Euclidean ball $B^n$, $j_{\frac n2-1,1}$ is the first positive zero of the Bessel function $J_{\frac n2-1}$ and $p_{\frac n2,1}$ is the first positive zero of the derivative of the Bessel function $J_{\frac n2}$. It is also well known that (see for instance  \cite[p. 61]{Henrot})
$$
{\l}^*_2(n) = 2^{\frac 2 n}{\l}^*_1(n). $$
The same relation is conjectured to hold true between ${\mu}^*_2(n)$ and ${\mu}^*_1(n)$ (see \cite{GNP} for a recent result about this conjecture in the 2-dimensional case). 
The following inequalities are also expected to be satisfied for every $k\ge 1$ (P\'olya's conjecture), 
$$  {\mu}^*_k(n)\le 4\pi^2 \left(\frac k{\omega_n} \right)^{\frac  2 n }\le {\l}^*_k(n), $$
where $4\pi^2 \left(\frac k{\omega_n} \right)^{\frac  2 n }$ is the first term of the  Weyl asymptotic expansion of both Dirichlet and Neumann eigenvalues of domains of volume one. Although this conjecture is still open, it was proved by Berezin \cite{Berezin} and Li and Yau \cite{LY1} that $ {\l}^*_k(n)\ge \frac n{n+2}4\pi^2 \left(\frac k{\omega_n} \right)^{\frac  2 n }$, while Kröger \cite{Kroger1,Kroger2} proved that $  {\mu}^*_k(n)\le \left(1+\frac n2\right)^{\frac  2 n }4\pi^2 \left(\frac k{\omega_n} \right)^{\frac  2 n }$.

\medskip

The first observation we make in this paper is that the sequence $\lambda^*_k(n)^{n/2}$  is subadditive while  $\mu^*_k(n)^{n/2}$ is superadditive. Indeed,  we prove (Theorem \ref{thm dirandneu}) that, for every $k\ge2$ and any finite family $i_1,\dots,  i_p$ of positive integers such that $i_1+ i_2+ \cdots + i_p=k$,
\begin{equation}\label{dir0}
{\l}^*_{k}(n)^{n/2}\le {\l}^*_{i_1}(n)^{n/2} + {\l}^*_{i_2}(n)^{n/2}+\cdots +{\l}^*_{i_p}(n)^{n/2}
\end{equation}
and
\begin{equation}\label{neu0}
{\mu}^*_{k}(n)^{n/2}\ge {\mu}^*_{i_1}(n)^{n/2} + {\mu}^*_{i_2}(n)^{n/2}+\cdots +{\mu}^*_{i_p}(n)^{n/2}.
\end{equation}
% In particular, for every $k\ge2$,$${\l}^*_{k+1}(n)^{n/2}- {\l}^*_{k}(n)^{n/2} \le {\l}^*_{1}(n)^{n/2}=j^2_{\frac n2-1,1} \omega_n^{\frac 2 n}  $$
% and 
% $${\mu}^*_{k+1}(n)^{n/2}- {\mu}^*_{k}(n)^{n/2} \ge {\mu}^*_{1}(n)^{n/2}=p^2_{\frac n2,1} \omega_n^{\frac 2 n}.$$
% $${\l}^*_{k+1}(2)-{\l}^*_k(2) \le \pi  j_{0,1}^2\approx 18.168$$ and $${\mu}^*_{k+1}(2)-{\mu}^*_k(2) \ge \pi p_{1,1}^2 \approx 10.65.$$

\medskip
An immediate consequence of Theorem \ref{thm dirandneu} and Fekete's Subadditive Lemma is that the sequences ${\l}^*_{k}(n)/k^{\frac 2n}$ and  ${\mu}^*_{k}(n)/k^{\frac 2n}$ are convergent and that P\'olya's conjecture for Dirichlet (resp. Neumann) eigenvalues is equivalent to the following  
$$\lim_k \frac{{\l}^*_{k}(n)}{k^{\frac 2n}} = 4\pi^2 \omega_n^{-\frac  2 n }$$
(resp. $\lim_k \frac{{\mu}^*_{k}(n)}{k^{\frac 2n}} = 4\pi^2 \omega_n^{-\frac  2 n }$, see Corollary \ref{Fekete}).

%\textbf{Fekete's Subadditive Lemma: For every subadditive sequence $\left \{ a_n \right \}_{n=1}^\infty$, the limit $\lim_{n \to \infty} \frac{a_n}{n}$ exists and is equal to $\inf \frac{a_n}{n}$.}

\medskip

Besides their theoretical interest, the inequalities \eqref{dir0} and \eqref{neu0} provide a ``rough test" for the numerical methods used to approximate  ${\l}^*_{k}(n)$ and ${\mu}^*_{k}(n)$. For example, we observe that the numerical values for ${\l}^*_k(2)$ obtained by Oudet  \cite{Oudet} (see also \cite[p. 83]{Henrot}) could be improved since the gap between the approximate values given for some successive ${\l}^*_k(2)$ exceeds $\pi  j_{0,1}^2$. Improvements of Oudet's calculations leading to approximate values which are consistent with \eqref{dir0} and \eqref{neu0} have been obtained recently  by Antunes and Freitas \cite{AnF}.

\medskip

Regarding the equality case in \eqref{dir0} we prove that if it holds, then the infimum ${\l}^*_k(n)$ is approximated to any desired accuracy by the $\l_k$ of a  disjoint union of $p$ domains $A_j$, $j=1,\dots, p$, each of which being, up to volume normalization,  an ``almost" minimizing domain for  ${\l}^*_{i_j}(n)$ (see Theorem \ref{thm dirandneu} for a precise statement). A similar phenomenon occurs for the case of equality in \eqref{neu0}.
 
%show that if the equality holds in  then, for every $\varepsilon >0$, the extreme value ${\l}^*_k(n)$ is achieved up to $\varepsilon$ by the disjoint union of $p$ domains each of them being, up to volume normalization,  an extremal domain etc. 

\medskip

%Alternative (Regarding  the equality case in \eqref{dir0}, we prove that if it holds, then there exists a minimizing sequence $\Omega_N$ for  ${\l}^*_k(n)$ such that each $\Omega_N$ is a disjoint union of $p$ domains $A^N_{i_j}$, $j=1,\dots, p$, with $A^N_{i_j}$ is, up to volume normalization,  a minimizing sequence of domains for  ${\l}^*_{i_j}(n)$. )

This result complements that by Wolf and Keller \cite{WK} where it is proved that if $\Omega=A\cup B$ is a disconnected minimizer of $\lambda_k$, then there exists a positive integer $i< k$ so that, after volume normalizations, $A$  minimizes  $\lambda_i$ and $B$ minimizes  $\lambda_{k-i}$ and, moreover, ${\l}^*_k(n)^{n/2}= {\l}^*_i(n)^{n/2} +{\l}^*_{k-i}(n)^{n/2}$. A Neumann analogue of this result has been recently obtained by Poliquin and Roy-Fortin \cite{Poliquin-Fortin}

\medskip

Our next observation is that Wolf-Keller's result extends to ``almost minimizing" disconnected domains as follows (Theorem \ref{WKDIR}): If a disconnected domain $\Omega=A\cup B$ minimizes $\lambda_k$ to within some $\varepsilon \ge 0$, then there exists an integer $i$ so that, after volume normalizations, $A$  minimizes  $\lambda_i^{n/2}$ to within $\varepsilon $ and $B$ minimizes  $\lambda_{k-i}^{n/2}$
to within $\varepsilon $, and, moreover,
$$0\le \left\{ {\l}^*_i(n)^{n/2} +{\l}^*_{k-i}(n)^{n/2} \right\} - {\l}^*_k(n)^{n/2}\le\varepsilon .$$

%$${\l}^*_i(n)^{n/2} +{\l}^*_{k-i}(n)^{n/2}\ge {\l}^*_k(n)^{n/2}\ge{\l}^*_i(n)^{n/2} +{\l}^*_{k-i}(n)^{n/2}-\varepsilon .$$
A similar property holds for ``almost maximizing" disconnected domains of Neumann eigenvalues (Theorem \ref{WKNEU}).

\medskip

The second part of the paper is devoted to the case of compact surfaces without boundary. If $S$ is an orientable compact surface of the 3-dimensional space, we denote by $\left\{\nu_k(S)\right\}_{k\ge 0}$ the spectrum of the Laplace-Beltrami operator acting on $S$ (here $\nu_0(S)=0$).  The eigenvalue $\nu_k$ is not bounded above on the set of compact surfaces of fixed area, as shown in \cite[Theorem 1.4]{CDE1} (which also justifies why we do not consider higher dimensional hypersurfaces). However, according to Korevaar \cite{K}, for every  integer  $\gamma \ge 0$, the $k$-th eigenvalue  $\nu_k$  is bounded above on the set $\mathcal M (\gamma)$ of compact  surfaces of genus $\gamma$ and fixed area. As before, we introduce the sequence 
$$  \nu^*_k (\gamma)=\sup \left\{\nu_k(S)\; : \; S\in \mathcal M (\gamma) \mbox{ and }|S|=1\right\}=\sup_{S\in \mathcal M (\gamma)}\nu_k(S)|S|.$$
As we will see in Section 3, an equivalent definition of $\nu^*_k (\gamma)$ consists in taking the supremum of the $k$-th eigenvalue $\nu_k(\Sigma_\gamma, g)$ of the Laplace-Beltrami operator on compact orientable 2-dimensional Riemannian manifolds  of  genus $\gamma$ and  area one.

\medskip

For $\gamma =0$, one has, from the results of Hersch  \cite{H} and  Nadirashvili \cite{N1} 
$$  \nu^*_1 (0) =8\pi  \quad \mbox{ and } \quad  \nu^*_2(0) =16\pi.$$
Results concerning extremal eigenvalues on surfaces of genus 1 and 2 can be found in \cite{EGJ, EI1, EI3, EI4, JLNNP, JNP, LY2, N}.  On the other hand, we have proved in \cite{CE} that the sequence $\nu^*_k (\gamma)$ is non decreasing with respect to $\gamma$ and that it is bounded below by a linear function of $k$ and $\gamma$.  A. Hassannezhad \cite{Asma} has recently proved that $\nu^*_k (\gamma)$ is also bounded from below by such a linear function of $k$ and $\gamma$. 

\medskip

In Theorem \ref{surfaces} we prove that the double sequence $\nu^*_k (\gamma)$  satisfies the following property (Theorem \ref{surfaces}): For every  $\gamma\ge 0$, $k\ge 1$, if  $\gamma_1\dots,\gamma_p \in\N$ and $i_1, \dots,i_p  \in\N^*$   are such that $\gamma_1+\dots+\gamma_p=\gamma $ and $i_1+ \dots+i_p=k$, then
 \begin{equation}\label{surface0}
 \nu_k^*(\gamma)\ge\nu_{i_1}^*(\gamma_1)+\dots +\nu_{i_p}^*(\gamma_p).
 \end{equation}

 As before, we investigate the equality case in \eqref {surface0} and establish the following Wolf-Keller's type result (Corollary \ref{WKSURF1}) :  Assume that the disjoint union $S_1 \sqcup S_2$ of two compact orientable surfaces $S_1$ and $S_2$  of genus $\gamma_1$, $\gamma_2$, respectively, satisfies 
 \begin{equation}
{\nu}_k(S_1 \sqcup S_2) = {\nu}_k^*(\gamma). 
\end{equation}
with $\vert S_1\vert+\vert S_2\vert =1$ and $\gamma_1+\gamma_2=\gamma$.  
Then there exists an integer $ i \in\{1,\cdots, k-1\} $ such that 
$${\nu}_k^*(\gamma) = {\nu}_i^*(\gamma_1)+{\nu}_{k-i}^* (\gamma_2) $$ 
$$\nu_i(S_1)|S_1|= {\nu}_i^*(\gamma_1)  \; \;\; \mbox{and} \;\; \; {\nu}_{k-i} (S_2)|S_2|= {\nu}_{k-i}^*(\gamma_2).$$

 \medskip
 
Actually, we give a more general result  where $S_1 \sqcup S_2$ is assumed to maximize $\nu_k$ to within a  positive $\varepsilon$ (Theorem \ref{WKSURF}). 

\medskip

Similar considerations can be made about nonorienrtable surfaces. This is discussed at the end of the paper. 

%Our next task is : Can an eigenvalue $\lambda_k$ be very large among domains with fixed $\lambda_1$ ? The answer is yes for Neumann eigenvalues since Colin De Verdière result.

\section{Dirichlet and Neumann eigenvalue problems on Euclidean domains}

To every (sufficiently regular) bounded domain $\Omega $ in $\R^n$, $n\ge 2$, we associate two sequences of real numbers 
$$
 0<\lambda_1(\Omega) \leq \lambda_2(\Omega) \leq \cdots \leq \lambda_k(\Omega) \leq \cdots 
$$
and $$
 0 = \mu_0(\Omega) < \mu_1(\Omega) \leq \mu_2(\Omega) \leq \cdots \leq \mu_k(\Omega) \leq \cdots 
$$
where $\lambda_k(\Omega)$ (resp. $\mu_k(\Omega)$) denotes the $k$-th eigenvalue of the Laplacian in $\Omega $ with Dirichlet (resp. Neumann) boundary conditions on $\partial \Omega$.
%It is well known that $\lambda_k(\Omega)$ is bounded below in terms of $k$ and the Euclidean volume $|\Omega|$ of $\Omega$, while $ \mu_k(\Omega) $ can be made arbitrarily small among domains of fixed volume. Conversely,  $\mu_k(\Omega)$ is bounded above in terms of $k$ and $|\Omega|$  while $ \sup\{\lambda_1(\Omega)\; ;\; |\Omega|=1\}=\infty $. We attach to $\R^n$ the two following sequences :
%$${\l}^*_k(n)=\inf \ \{\l_k(\Omega)\ ;\ \Omega\subset \R^n;\ |\Omega|=1\}$$and
%$${\mu}^*_k(n)=\sup \ \{\mu_k(\Omega)\ ;\ \Omega\subset \R^n;\ |\Omega|=1\}.$$
If  $t$ is a positive number, the notation $t\ \Omega$ will designate  the image of the domain $\Omega$ under the Euclidean dilation of ratio $t$.  One has
$$\l_k(t\ \Omega)=t^{-2}\l_k(\Omega)\; , \; \mu_k(t\ \Omega)=t^{-2}\mu_k(\Omega)\; \mbox{and} \; |t\ \Omega|=t^n|\Omega|$$
and, then
\begin{eqnarray}\label{lambda}
\nonumber {\l}^*_k(n)&=&\inf \ \{\l_k(\Omega)\ :\ \Omega\subset \R^n,\ |\Omega|=1\} \\
&=&\inf \ \{\l_k(\Omega)\ :\ \Omega\subset \R^n,\ |\Omega|\le 1\}\\
\nonumber &=&\inf\ \{\l_k(\Omega)\vert \Omega\vert^{2/n}\ : \ \Omega \subset \R^n\}
\end{eqnarray}
and
\begin{eqnarray}\label{mu}
\nonumber{\mu}^*_k(n)&=&\sup \ \{\mu_k(\Omega)\ :\ \Omega\subset \R^n,\ |\Omega|=1\}\\
&=&\sup \ \{\mu_k(\Omega)\ :\ \Omega\subset \R^n,\ |\Omega|\ge 1\}\\
\nonumber &=&\sup\ \{\mu_k(\Omega)\vert \Omega\vert^{2/n}\ :\ \Omega \subset \R^n\}.
\end{eqnarray}

\medskip

The sequences ${\l}^*_k(n)$ and ${\mu}^*_k(n)$ satisfy the following intrinsic properties. 

\begin{thm}\label{thm dirandneu}
Let $n$ and $k$ be two positive integers and let $i_1\le i_2\le \cdots\le i_p$ be positive integers such that $i_1+ i_2+ \cdots + i_p=k$. 

\noindent 1) We have, 
\begin{equation}\label{dir}
{\l}^*_{k}(n)^{n/2}\le {\l}^*_{i_1}(n)^{n/2} + {\l}^*_{i_2}(n)^{n/2}+\cdots +{\l}^*_{i_p}(n)^{n/2}
\end{equation}
and
\begin{equation}\label{neu}
{\mu}^*_{k}(n)^{n/2}\ge {\mu}^*_{i_1}(n)^{n/2} + {\mu}^*_{i_2}(n)^{n/2}+\cdots +{\mu}^*_{i_p}(n)^{n/2},
\end{equation}

\medskip

\noindent 2) If the equality holds in \eqref{dir}, then, for every $\varepsilon >0$, there exist $p$ mutually disjoint domains $A_1, A_2, \cdots, A_p$ such that

\smallskip

\begin{itemize}
\item [i)] ${\l}_{k} (A_1\cup \cdots\cup A_p) \le (1+\varepsilon){\l}^*_{k} $ ;
		
	\item [ii)] $\forall j\le p$, ${\l}^*_{i_j}\le{\l}_{i_j}( A_j) |A_j|^{2/n} \le (1+\varepsilon){\l}^*_{i_j}$.
		\item [iii)] $|A_1| +\cdots +|A_p|=1$ and, $\forall j\le p$, $\frac{{\l}^*_{i_j}}{(1+\varepsilon){\l}^*_{k}}\le |A_j|^{2/n}\le \frac{(1+\varepsilon){\l}^*_{i_j}}{{\l}^*_{k}}$; 

\end{itemize}
where ${\l}^*_{k}$ stands for ${\l}^*_{k}(n)$. 
\medskip

\noindent 3)
If the equality holds in \eqref{neu}, then, for every $\varepsilon >0$, there exist $p$ mutually disjoint domains $A_1, A_2, \cdots, A_p$ such that

\smallskip

\begin{itemize}
\item [i)] ${\mu}_{k} (A_1\cup \cdots\cup A_p) \ge (1-\varepsilon){\mu}^*_{k} $ ;
		
	\item [ii)] $\forall j\le p$, $(1-\varepsilon){\mu}^*_{i_j}\le{\mu}_{i_j}( A_j) |A_j|^{2/n} \le {\mu}^*_{i_j}$.
	\item [iii)] $|A_1| +\cdots +|A_p|=1$ and, $\forall j\le p$, $ \frac{(1-\varepsilon){\mu}^*_{i_j}}{{\mu}^*_{k}}\le |A_j|^{2/n}\le \frac{{\mu}^*_{i_j}}{(1-\varepsilon){\mu}^*_{k}}$; 
\end{itemize}
where ${\mu}^*_{k}$ stands for ${\mu}^*_{k}(n)$. 
\end{thm}

\begin{proof}
Let $\varepsilon$ be any positive real number. For each $j\le p$, let $C_j$ be a domain of volume 1  satisfying  $${\l}^*_{i_j}(n)\le{\l}_{i_j}( C_j) \le (1+\varepsilon){\l}^*_{i_j}(n)$$ and set $B_j=\left( {\l}_{i_j}( C_j)/ {\l}^*_{k}(n)\right)^{\frac 1 2} C_j $ so that  
$${\l}_{i_j}( B_j)= {\l}^*_{k}(n)\quad \mbox{and} \quad |B_j| = \left( {\l}_{i_j}( C_j)/ {\l}^*_{k}(n)\right)^{\frac n 2}.$$

 One can assume w.l.o.g. that the domains $B_1,  \cdots, B_p$ are mutually disjoint. 
Let us introduce  the domain $\Omega=B_1\cup \cdots\cup B_p$. Since for every $j\le p$, ${\l}_{i_j}( B_j)= {\l}^*_{k}(n)$ and since the spectrum of $\Omega$ is the union of the spectra of the $B_j$'s, one has 
$$\#\left\{l\in\N^*\ ;\ {\l}_{l}( \Omega)\le {\l}^*_{k}(n)\right\} = \sum_{j=1}^p\#\left\{l\in\N^*\ ;\ {\l}_{l}( B_j)\le {\l}^*_{k}(n)\right\} \ge \sum_{j=1}^p i_j =k.$$
Thus, ${\l}_{k}(\Omega)\le{\l}^*_{k}(n)$. Since ${\l}^*_{k}(n)\le {\l}_{k}(\Omega) |\Omega|^{\frac  2n}$, the volume of $\Omega$ should be greater than or equal to 1. Consequently,
\begin{equation}\label{dir'}
 1\le |\Omega| =\sum_{j\le p}|B_j|= \frac 1{{\l}^*_{k}(n)^{\frac n 2}}\sum_{j\le p} {\l}_{i_j}( C_j)^{\frac n 2}\le \frac {(1+\varepsilon)^{\frac n 2}}{{\l}^*_{k}(n)^{\frac n 2}}\sum_{j\le p} {\l}^*_{i_j}(n)^{\frac n 2} .
 \end{equation}
Inequality \eqref{dir} follows immediately from \eqref{dir'} since $\varepsilon$ can be  arbitrarily small. 

Assume now that the equality holds in \eqref{dir} and consider for each positive $\varepsilon$, a  family $B_1, B_2, \cdots, B_p$ constructed as above. Using \eqref{dir'}, one sees that the domain $\Omega=B_1\cup B_2\cup \cdots\cup B_p$ satisfies $1\le |\Omega|\le {(1+\varepsilon)^{\frac n 2}}$ and it is easy to check that  the domains $A_j:=|\Omega|^{-\frac 1n}B_j$, $j\le p$, 
satisfy  the properties (ii) and (iii) of the statement (indeed, $|A_j|=\frac {|B_j|}{|\Omega|}$ with $ \left( {\l}^*_{i_j}(n)/ {\l}^*_{k}(n)\right)^{\frac n 2}\le |B_j|\le \left( (1+\varepsilon){\l}^*_{i_j}(n)/ {\l}^*_{k}(n)\right)^{\frac n 2} $).  As for (i), one has for each $j\le p$, ${\l}_{i_j}( A_j)= |\Omega|^{\frac 2n}{\l}^*_{k}(n)$. Since $k=i_1+ i_2+ \cdots + i_p$, one deduces that  ${\l}_{k} (A_1\cup A_2\cup \cdots\cup A_p)=|\Omega|^{\frac 2n}{\l}^*_{k}(n)\le (1+\varepsilon){\l}^*_{k} (n)$. 

\medskip

The proof  in the Neumann case follows the same outline. Indeed, for any positive $\varepsilon$, we consider $p$ mutually disjoint domains $C_1, C_2, \cdots, C_p$ of volume 1 such that, $\forall j\le p$,  $${\mu}^*_{i_j}(n)\ge{\mu}_{i_j}( C_j) \ge (1-\varepsilon){\mu}^*_{i_j}(n)$$ and set $B_j=\left( {\mu}_{i_j}( C_j)/ {\mu}^*_{k}(n)\right)^{\frac 1 2} C_j $ and  $\Omega=B_1\cup B_2\cup \cdots\cup B_p$.  
Since for every $j\le p$, ${\mu}_{i_j}( B_j)= {\mu}^*_{k}(n)$, the number of eigenvalues of $B_j$ that are \emph{strictly} less than ${\mu}^*_{k}(n)$ is at most $i_j$ (recall that ${\mu}_{i_j}( B_j)$ denotes the $(i_j+1 )$-th eigenvalue of $ B_j$). As the spectrum of $\Omega$ is the union of the spectra of the $B_j$'s, it is clear that the number of eigenvalues of $\Omega$ that are strictly less than ${\mu}^*_{k}(n)$ is at most $k=i_1+ i_2+ \cdots + i_p$. Thus, ${\mu}_{k}(\Omega)\ge{\mu}^*_{k}(n)$ which implies (since ${\mu}^*_{k}(n)\ge {\mu}_{k}(\Omega)|\Omega|^{\frac 2n}$) that the volume of $\Omega$ is less than or equal to 1.  To derive Inequality \eqref{neu} it suffices to observe  that $1\ge |\Omega| =\sum_{j\le p}|B_j|$ and that $|B_j| = \left( {\mu}_{i_j}( C_j)/ {\mu}^*_{k}(n)\right)^{\frac n 2}\ge {(1-\varepsilon)^{\frac n 2}}\frac  {{\mu}^*_{i_j}(n)^{\frac n 2}}{{\mu}^*_{k}(n)^{\frac n 2}} .
$

Assume now that the equality holds in \eqref{neu} and consider for each positive $\varepsilon$, a  family $B_1, B_2, \cdots, B_p$ constructed as above. The domain $\Omega=B_1\cup B_2\cup \cdots\cup B_p$ satisfies $1\ge |\Omega|\ge {(1-\varepsilon)^{\frac n 2}}$ and it is easy to check that  the domains $A_j:=|\Omega|^{-\frac 1n}B_j$, $j\le p$, 
satisfy  the properties (ii) and (iii) of the statement (indeed, $|A_j|=\frac {|B_j|}{|\Omega|}$ with $\left( (1-\varepsilon){\mu}^*_{i_j}(n)/ {\mu}^*_{k}(n)\right)^{\frac n 2} \le |B_j|\le \left( {\mu}^*_{i_j}(n)/ {\mu}^*_{k}(n)\right)^{\frac n 2} $).  Moreover, one has for each $j\le p$, ${\mu}_{i_j}( A_j)= |\Omega|^{\frac 2n}{\mu}^*_{k}(n)$. Thus, ${\mu}_{k} (A_1\cup A_2\cup \cdots\cup A_p)=|\Omega|^{\frac 2n}{\mu}^*_{k}(n)\ge (1-\varepsilon){\mu}_{k} $ which proves (i). 

\end{proof}

\begin{cor}\label{successivegap}
For every $n\ge 2$ and every $k\ge 1$, we have 
$$
{\l}^*_{k+1}(n)^{n/2}-{\l}^*_k(n)^{n/2} \le{\l}^*_1(n)^{n/2}= j^n_{\frac n2-1,1} \omega_n
$$
and
$$
{\mu}^*_{k+1}(n)^{n/2}-{\mu}^*_k(n)^{n/2} \ge{\mu}^*_1(n)^{n/2}=p^n_{\frac n2,1} \omega_n.
$$

\end{cor}

\begin{rem}
(i) The first inequality in Corollary \ref{successivegap} is sharp for  $k=1$ since we know that  ${\l}^*_2(n)=2^{2/n}{\l}^*_1(n)$. 

\noindent (ii) In dimension 2, the inequalities of Corollary \ref{successivegap}  lead to 
 $${\l}^*_{k+1}(2)-{\l}^*_k(2) \le \pi  j_{0,1}^2\approx 18.168$$
 and
 $${\mu}^*_{k+1}(2)-{\mu}^*_k(2) \ge \pi p_{1,1}^2 \approx 10.65,$$
 which provides a simple tool to test the accuracy of  numerical approximations. %For instance, the inequality for ${\l}^*_k(2)$  shows that the numerical values  given in  \cite[p.83]{Henrot} could be improved since the gap between the values given for ${\l}^*_7(2)$ and ${\l}^*_6(2)$ exceeds $\pi  j_{0,1}^2$. 

\noindent (iii) Iterating the inequalities of Corollary \ref{successivegap}  we get 
$${\l}^*_k(n) \le  j^2_{\frac n2-1,1} \omega_n^{2/n} k^{2/n}
$$
and
$${\mu}^*_k(n) \ge p^2_{\frac n2,1} \omega_n^{2/n} k^{2/n}.$$
Combining these inequalities with P\'olya conjecture, we expect the following estimates
$$ p^2_{\frac n2,1} \omega_n^{2/n} k^{2/n}\le {\mu}^*_k(n)\le 4\pi^2 \left(\frac k{\omega_n} \right)^{\frac  2 n }\le {\l}^*_k(n) \le  j^2_{\frac n2-1,1} \omega_n^{2/n} k^{2/n}
$$
%and $$  p^2_{\frac n2,1} \omega_n^{2/n} k^{2/n}\le {\mu}^*_k(n) \le 4\pi^2 \left(\frac k{\omega_n} \right)^{\frac  2 n }.$$
which take the following form in dimension 2 : 
$$  4\ \pi k \le {\l}^*_k(2) \le  5.784\ \pi k
$$
and $$  3.39\ \pi k\le {\mu}^*_k(2) \le 4\ \pi k .$$

\noindent (iv) Let $\Omega\subset \R^n$ be the union of $k$ balls of the same radius $r=(k\omega_n)^{-n}$ so that $|\Omega|=1$.
Then
$$
\lambda_k(\Omega)=\lambda_1(B^n)=\lambda_1(B^n)(k\omega_n)^{2/n},
$$
and
$$
\lambda_{k+1}(\Omega)=\lambda_2(B^n)=\lambda_2(B^n)(k\omega_n)^{2/n}.
$$
Thus,
$$
\lambda_{k+1}(\Omega)^{n/2}-\lambda_k(\Omega)^{n/2}=k\omega_n\left(\lambda_2(B^n)^{n/2}-\lambda_1(B^n)^{n/2}\right).
$$
This shows that the gap  $\lambda_{k+1}(\Omega)^{n/2}-\lambda_k(\Omega)^{n/2}$  cannot be  bounded independently of $k$ (see also Proposition \ref{gaps} below). Corollary \ref{successivegap} tells us that such a bound exists  when we consider the sequence of infima of $\l_k$.

\end{rem}

Thanks to Fekete's Lemma, 
the subadditivity of the sequence ${\l}^*_k(n)^{n/2}$ leads immediately to the following corollary. 
\begin{cor}\label{Fekete}
For every $n\ge 2$, the sequence $ \frac { {\l}^*_k(n)}{k^{2/n}}$ converges to a positive limit  with 
$$\lim_k\frac { {\l}^*_k(n)}{k^{2/n}} =\inf_k \frac { {\l}^*_k(n)}{k^{2/n}}.$$
In particular, the two following properties are equivalent :
\begin{enumerate}
\item (P\'olya's conjecture) For every $k\ge 1$ and every domain $\Omega \subset\R^n$, 
$$\l_k(\Omega)\ge 4\pi^2 (\vert \Omega\vert\omega_n)^{-2/n} k^{2/n}$$
\item $\displaystyle \lim_k\frac { {\l}^*_k(n)}{k^{2/n}} =4\pi^2 \omega_n^{-2/n} .$
\end{enumerate}

\end{cor}

A similar result holds for the Neumann Laplacian eigenvalues. 

\medskip

The inequality \eqref{dir}  leads to
\begin{equation}\label{dir''}
{\l}^*_k(n)^{n/2}\le\inf_{1\le i\le k-1} \left\{{\l}^*_i(n)^{n/2}+{\l}^*_{k-i}(n)^{n/2}\right\}.
\end{equation}
Wolf and Keller \cite{WK} proved that if ${\l}_k$ is minimized by a non connected domain, that is  $ {\l}^*_k(n)={\l}_k(A \cup B)$ for a couple of disjoint domains $A$ and $B$ with $|A|>0$, $|B|>0$ and $|A|+|B| =1$, then the equality holds in \eqref{dir''} and, moreover, $A$ and $B$ are, up to normalizations,  minimizers of $\lambda_i$ and ${\l}_{k-i} $, respectively.  
The Neumann's analogue of this result has been established by Poliquin and Roy-Fortin \cite{Poliquin-Fortin}.

The following theorem shows how Wolf-Keller's result extends to ``almost minimizing" disconnected domains.

\begin{thm}\label{WKDIR} Let $k\ge 2$ and assume that there exists a non connected domain $\Omega =A \cup B$ in $\R^n$ with $|A|+|B| =1$,  $|A|> {\varepsilon}/{{\l}^*_k(n)^{n/2}}$, $|B|> {\varepsilon}/{{\l}^*_k(n)^{n/2}}$ and  
\begin{equation}\label{WKdir}
{\l}_k(A \cup B)^{n/2} \le {\l}^*_k(n)^{n/2} +\varepsilon
\end{equation}
for some $\varepsilon \ge 0 $.
Then there exists an integer $ i \in\{1,\cdots, k-1\} $ such that 
$$0\le \left\{ {\l}^*_i(n)^{n/2}+{\l}^*_{k-i}(n)^{n/2}\right\} - {\l}^*_k(n)^{n/2}  \le \varepsilon ,$$ 
$$0\le\l_i(A)^{n/2}|A| - {\l}^*_i(n)^{n/2}\le \varepsilon \; \;\; \mbox{and} \;\; \; 0\le{\l}_{k-i} (B)^{n/2}|B| - {\l}^*_{k-i}(n)^{n/2} \le 
\varepsilon. $$

%$${\l}^*_k(n)^{n/2} \le {\l}^*_i(n)^{n/2}+{\l}^*_{k-i}(n)^{n/2} \le {\l}^*_k(n)^{n/2} +\varepsilon ,$$ 
%$$\l_i(A)^{n/2}|A|\le {\l}^*_i(n)^{n/2}+\varepsilon \; \;\; \mbox{and} \;\; \; {\l}_{k-i} (B)^{n/2}|B|\le {\l}^*_{k-i}(n)^{n/2}+ \varepsilon. $$
\end{thm}
\begin{proof}
Since the spectrum of $\Omega=A\cup B $ is the re-ordered union of the spectra of $A$ and $B$, the eigenvalue $\l_k(\Omega)$ belongs to the union of the spectra of $A$ and $B$ and, moreover, 
\begin{equation}\label{WKdir1}
 \#\left\{j\in \N^*\; ; \; \l_j(A)<\l_k(\Omega )\right\}+ \#\left\{j\in \N^*\; ; \; \l_j(B)<\l_k(\Omega )\right\} \le k-1
 \end{equation}
and
\begin{equation}\label{WKdir2}
 \#\left\{j\in \N^*\; ; \; \l_j(A)\le\l_k(\Omega)\right\}+ \#\left\{j\in \N^*\; ; \;  \l_j(B)\le\l_k(\Omega )\right\} \ge k .
 \end{equation}
 Hence, there exists at least one integer $j\in\{1,\dots, k\}$ such that $\l_j(A)=\l_k(\Omega )$ or $\l_j(B)=\l_k(\Omega )$. Assume that the first alternative occurs and let $i$ be the largest integer between $1$ and $k$ such that $\l_i(A)=\l_k(\Omega)$. 
 
 Observe first that $i\le k-1$. Indeed, if $\l_k(A)=\l_k(\Omega)$, then 
 $${\l}^*_k(n)^{n/2}\le\lambda_k (A)^{n/2}|A|=\l_k(\Omega)^{n/2}|A|\le \left({\l}^*_k(n)^{n/2}+\varepsilon\right)|A|$$
which implies  $ |A|\ge \frac {{\l}^*_k(n)^{n/2}}{{\l}^*_k(n)^{n/2}+\varepsilon}$ and, then
$|B|=1-|A|\le \frac {\varepsilon}{{\l}^*_k(n)^{n/2}+\varepsilon}\le \frac {\varepsilon}{{\l}^*_k(n)^{n/2}}$. This contradicts the volume assumptions of the theorem. 
 
On the other hand,  the maximality of $i$ means that $$\#\left\{j\in \N^*\; ; \; \l_j(A)\le\l_k(\Omega)\right\}=i$$ which implies, thanks  to \eqref{WKdir2}, 
 $\l_{k-i}(B)\le \l_k(\Omega )$. %since, otherwise, we would have thanks to \eqref{WKdir2}, $\l_{i+1}(A)=\l_k(\Omega )$ which contradicts the maximality of $i$. 
 Thus,  
\begin{equation}\label{WKdir3}
 \l_k(\Omega )^{n/2}=\l_k(\Omega )^{n/2}|A| +\l_k(\Omega )^{n/2} |B|\ge \l_i(A)^{n/2}|A|+\l_{k-i}(B)^{n/2}|B|. 
\end{equation}
Since $\l_i(A)^{n/2}|A|\ge {\l}^*_i(n)^{n/2}$ and $\l_{k-i}(B)^{n/2}|B|\ge {\l}^*_{k-i}(n)^{n/2}$, we have proved the inequality
$$
{\l}^*_k(n)^{n/2} +\varepsilon\ge \l_k(\Omega )^{n/2}\ge {\l}^*_i(n)^{n/2}+{\l}^*_{k-i}(n)^{n/2}. 
$$

Now, we necessarily have the inequality $\l_i(A)^{n/2}|A|\le {\l}^*_i(n)^{n/2}+\varepsilon$. Otherwise, we would have, thanks to \eqref{WKdir3} and Theorem  \ref{thm dirandneu},
\begin{eqnarray*}
\l_k(\Omega )^{n/2}\ge \l_i(A)^{n/2}|A|+\l_{k-i}(B)^{n/2}|B|  & > &  {\l}^*_i(n)^{n/2}+\varepsilon + {\l}^*_{k-i}(n)^{n/2}\\
&\ge& {\l}^*_k(n)^{n/2}+\varepsilon 
\end{eqnarray*}
which contradicts the assumption of the theorem. 
The same argument leads to the inequality ${\l}_{k-i} (B)^{n/2}|B|\le {\l}^*_{k-i}(n)^{n/2}+ 
\varepsilon $.
\end{proof}

\begin{rem}
(i) Taking $\varepsilon =0$ in Theorem \ref{WKDIR}, all the inequalities of the theorem become equalities and  we recover the result of Wolf and Keller. Notice that when $\varepsilon =0$, it is immediate to see that the integer $i$ is such that ${\l}^*_i(n)^{n/2}+ {\l}^*_{k-i}(n)^{n/2}$ is minimal. 

\medskip

\noindent (ii) The assumption that the volume of each of the components $A$ and $B$  of 
$\Omega$ is bounded below in terms of $\varepsilon$ is necessary to guarantee that the integer $i$ is different from $0$ and $k$ in Theorem \ref{WKDIR}. Indeed, take for $A$ a domain whose volume is almost equal to one and such that $\l_k(A)^{n/2}\le {\l}^*_k(n)^{n/2}+\varepsilon$, and take for $B$ a domain of small volume such that $\l_1(B)^{n/2}>{\l}^*_k(n)^{n/2}+\varepsilon$. The domain $\Omega=A\cup B$ would have volume one and   $\l_k(\Omega)=\l_k(A)<\l_1(B)$.

%If we add the assumption that the volume of each of the components $A$ and $B$  of $\Omega$ exceeds $\frac {\varepsilon}{{\l}^*_k(n)^{n/2}}$, then we can  guarantee that the integer $i$ is different from $0$ and $k$. Indeed,  $i=0$ means that  $\lambda_k (B)= \l_k(A\cup B)$ which implies $${\l}^*_k(n)^{n/2}\le\lambda_k (B)^{n/2}|B|\le \left({\l}^*_k(n)^{n/2}+\varepsilon\right)|B|$$and, then $|A|=1-|B|\le \frac {\varepsilon}{{\l}^*_k(n)^{n/2}+\varepsilon}\le \frac {\varepsilon}{{\l}^*_k(n)^{n/2}}$. Of course, $i=k$ implies the same inequality but for the volume of $B$.
\end{rem}

%\texttt{Question : Que peut-on dire des fonctionnelles :$$\Lambda_k=\sum_{i=1}^k \l_i^{\frac 2n}$$ et $$M_k=\sum_{i=1}^k \mu_i^{\frac 2n} ??$$}

Using similar arguments as in the proof of Theorem \ref{WKDIR} (see also the proof of 
 Theorem \ref{WKSURF}), we obtain the following 
 
 \begin{thm}\label{WKNEU} Let $k\ge 2$ and assume that there exists a non connected domain $\Omega =A \cup B$ in $\R^n$ with $|A|+|B| =1$ and  
\begin{equation}\label{WKneu}
{\mu}_k(A \cup B)^{n/2} \ge {\mu}^*_k(n)^{n/2} -\varepsilon
\end{equation}
for some $\varepsilon \ge 0 $.
Then there exists an integer $ i \in\{1,\cdots, k-1\} $ such that 
$$0\le {\mu}^*_k(n)^{n/2} -\left[{\mu}^*_i(n)^{n/2}+{\mu}^*_{k-i}(n)^{n/2} \right]\le \varepsilon ,$$ 
$$0\le {\mu}^*_i(n)^{n/2} - \mu_i(A)^{n/2}|A|\le \varepsilon \; \;\; \mbox{and} \;\; \; 0\le  {\mu}^*_{k-i}(n)^{n/2} - {\mu}_{k-i} (B)^{n/2}|B|\le
\varepsilon. $$

%$${\mu}^*_k(n)^{n/2} \ge {\mu}^*_i(n)^{n/2}+{\mu}^*_{k-i}(n)^{n/2} \ge {\mu}^*_k(n)^{n/2} -\varepsilon ,$$ 
%$$\mu_i(A)^{n/2}|A|\ge {\mu}^*_i(n)^{n/2}-\varepsilon \; \;\; \mbox{and} \;\; \; {\mu}_{k-i} (B)^{n/2}|B|\ge {\mu}^*_{k-i}(n)^{n/2}-\varepsilon. $$
\end{thm}

\begin{rem}
(i) Taking $\varepsilon =0$ in Theorem \ref{WKNEU}, all the inequalities of the theorem become equalities and the integer $i$ is necessarily such that ${\mu}^*_i(n)^{n/2}+ {\mu}^*_{k-i}(n)^{n/2}$ is maximal. 

\medskip

\noindent (ii) A consequence of Theorem \ref{WKNEU} is that if for some $\varepsilon >0$, there exists a domain $\Omega$ in $\R^n$ with
$$\mu_k(\Omega)^{n/2}> \sup_{1\le i\le k-1} \left\{{\mu}^*_i(n)^{n/2}+ {\mu}^*_{k-i}(n)^{n/2}\right\} +\varepsilon ,$$
then ${\mu}^*_k(n)$ cannot be approached up to $\varepsilon$ by a non connected domain.

\end{rem}

The following properties are likely well known, we show them here for completeness and comparison with other results in this section.

\begin{prop}\label{gaps}
For every $n\ge 2$ and $k\ge 1$   we have 
\begin{equation}\label{infgapdir}
\inf\{\l_{k}(\Omega)-\l_1(\Omega) :\; \Omega\subset \R^n, \; \vert \Omega\vert=1 \}=0\; ;
\end{equation}
\begin{equation}\label{supgapdir}
\sup\{\l_{k+1}(\Omega)-\l_k(\Omega)\; :\; \Omega\subset \R^n, \; \vert \Omega\vert=1\}=\infty \; ;
\end{equation}
\begin{equation}\label{infgapneu}
\inf\{\mu_{k}(\Omega)-\mu_1(\Omega) :\; \Omega\subset \R^n, \; \vert \Omega\vert=1 \}=0\ ;
\end{equation}
\begin{equation}\label{supgapneu}
{\mu}^*_{1}(n)(k+1)^{\frac 2 n}\le\sup\{\mu_{k+1}(\Omega)-\mu_k(\Omega)\; :\; \Omega\subset \R^n, \; \vert \Omega\vert=1\}\le {\mu}^*_{k+1}(n).
 \end{equation}
\end{prop}

\begin{proof} To see \eqref{infgapdir} it suffices to  consider a domain $\Omega$ modeled on the disjoint union of $k+1$ identical balls of volume $\frac 1{k+1}$. The $k+1$ first Dirichlet eigenvalues of such a domain are almost equal.

  Now, 
take any domain $D$ with $\l_{k+1}(D)-\l_k(D)>0$ and observe that   
$ \l_{k+1}(t D)-\l_k(t D) \to +\infty$ as $t\to 0$. 
Then attach to the domain $tD$ a sufficiently long and thin domain in order to obtain a volume $1$ domain $\Omega(t)$ with $\l_k(\Omega (t))\approx\l_k(t D)$ and $\l_{k+1}(\Omega (t))\approx\l_{k+1}(t D)$ (recall that the first eigenvalue of a box of volume 1 goes to infinity as the length of one of its sides becomes very small). Thus, $ \l_{k+1}(\Omega (t))-\l_k(\Omega (t))$ goes to infinity as $t\to 0$ which proves \eqref{supgapdir}.

  As for the Neumann eigenvalues of a domain $\Omega$ modeled on the disjoint union of $k+1$ identical balls of volume $\frac 1{k+1}$, one has $\mu_0(\Omega)=0$ and $\mu_1(\Omega),\cdots ,\mu_k(\Omega)$ are almost equal to zero, while $\mu_{k+1}(\Omega)$ is almost equal to the first positive eigenvalue of one of the balls, that is $\mu_{k+1}(\Omega)\approx {\mu}^*_{1}(n)(k+1)^{\frac 2 n}$. This example proves \eqref{infgapneu} and \eqref{supgapneu}.

%Regarding \eqref{supgapneu}, Kröger \cite{} prove that $$\hat\mu_k(n)\le \left(\frac{n+2}2\right)^{\frac 2 n} 4\pi^2 \left(\frac k{\omega_n} \right)^{\frac  2 n }.$$ 

\end{proof}

\section{Eigenvalues of closed surfaces}
There are two equivalent approaches to introduce the extremal eigenvalues  on closed surfaces. 

Let us start with the ``embedded" point of view. Indeed, 
if $S$ is a  compact connected surface of the 3-dimensional Euclidean space $\R^3$, we  consider on it the Dirichlet's energy functional associated with the tangential gradient, and denote by
$$
 0 = \nu_0(S) < \nu_1(S) \leq \nu_2(S) \leq \cdots \leq \nu_k(S) \leq \cdots .
$$
 the spectrum of the corresponding Laplacian.  %For every positive $t$, the notation $tS$ designates the image of $S$ under the dilation of 
 According to \cite[Theorem 1.4]{CDE1}, one has, $\forall k\ge 1$,
 $$\sup_{\vert S\vert=1}\nu_k(S) =+\infty.$$
However, it is known since the work of Korevaar \cite{K} that for every  integer  $\gamma \ge 0$, the $k$-th eigenvalue  $\nu_k$  is bounded above on the set of compact  surfaces of genus $\gamma$. Thus,   for every  integer  $\gamma \ge 0$ we denote by $\mathcal M (\gamma)$  the set of all compact  surfaces of genus $\gamma$ embedded in $\R^3$ and define the sequence 
$$  \nu^*_k (\gamma)=\sup \left\{\nu_k(S)\; ; \; S\in \mathcal M (\gamma) \mbox{ and }|S|=1\right\}=\sup_{S\in \mathcal M (\gamma)}\nu_k(S)|S|,$$
where $|S|$ stands for the area of $S$. 
Regarding the infimum, it is well known that  $\inf_{S\in \mathcal M (\gamma)}\nu_k(S)|S|=0$.

\medskip

Alternatively, let $\Sigma_{\gamma}$ be an abstract closed orientable 2-dimensional smooth manifold of genus $\gamma$. To every Riemannian metric $g$ on $\Sigma_{\gamma}$ we  associate the sequence of eigenvalues of the Laplace-Beltrami operator $\Delta_g$ 
$$
 0 = \nu_0(\Sigma_{\gamma},g) < \nu_1(\Sigma_{\gamma},g) \leq \nu_2(\Sigma_{\gamma},g) \leq \cdots \leq \nu_k(\Sigma_{\gamma},g) \leq \cdots .
$$
Notice that for every positive number $t$, one has $\nu_k(\Sigma_{\gamma},tg) = t^{-1}\nu_k(\Sigma_{\gamma},g)$ while the Riemannian area satisfies $|(\Sigma_{\gamma},tg)|=t|(\Sigma_{\gamma},g)|$ so that the product 
$\nu_k (\Sigma_{\gamma},g) |(\Sigma_{\gamma},g)|$ is invariant under scaling of the metric. 
%It is well known that (see \cite{})
%$$\inf_g \nu_k (\Sigma_{\gamma},g) |(M,g)|^{2/n}=\inf\left\{\nu_k (g)\ ; \  |(M,g)|=1\right\}=0.$$
%On the other hand, when the dimension of $M$ is at least 3, then (see \cite{})
%$$\sup_g \nu_k (g) |(M,g)|^{2/n}=\sup\left\{\nu_k (g)\ ; \  |(M,g)|=1\right\}=\infty.$$

 \begin{lemme}\label{C1-isometric}
 Let $\Sigma_{\gamma}$ be a closed orientable 2-dimensional smooth manifold of genus $\gamma \ge 0$ and denote by $ \mathcal R (\Sigma_\gamma)$ the set of all Riemannian metrics on $\Sigma_\gamma$. For every positive integer $k$ one has
 \begin{eqnarray*} 
 \nu^*_k (\gamma) &=& \sup \left\{\nu_k (\Sigma_{\gamma},g)\, ; \, g\in \mathcal R (\Sigma_\gamma)  \mbox{ and }|(\Sigma_{\gamma},g)|=1\right\} \\
& = &\sup_{g\in\mathcal R (\Sigma_\gamma) } \nu_k (\Sigma_{\gamma},g)\vert(\Sigma_{\gamma},g)\vert.
 \end{eqnarray*} 
 \end{lemme}
 \begin{proof}
 Let us first recall the well-known fact (see e.g.   Dodziuk's paper  \cite{D}) that if two Riemannian metrics $g_1$, $g_2$ on a compact manifold $M$ of dimension $m$ are quasi-isometric with a quasi-isometry ratio close to $1$, then the spectra of their Laplacians  are  close. More precisely, we say that $g_1$ and $g_2$ are $\alpha$-quasi-isometric, with $\alpha \ge 1$, if for each $v\in TM$, $v\not =0$, we have
$$
\frac{1}{\alpha^2}\le \frac{g_1(v,v)}{g_2(v,v)} \le \alpha^2.
$$
The spectra of $g_1$ and $g_2$ then satisfy, $\forall k\ge 1$, 
 \begin{eqnarray} \label{quasi-eigenvalue}
\frac{1}{\alpha^{2(m+1)}}\le \frac{\nu_k(M,g_1)}{\nu_k(M,g_2)} \le \alpha^{2(m+1)}
\end{eqnarray} 
while the ratio of  their volumes is so that
 \begin{eqnarray} \label{quasi-volume}
\frac{1}{\alpha^m}\le \frac{\vert(M,g_1)\vert}{\vert (M,g_2)\vert} \le \alpha^m.
\end{eqnarray} 
 Now, any surface $S\in \mathcal M (\gamma)$ is of the form 
 $S=\phi (\Sigma_\gamma) $,
 where $\phi:\Sigma_\gamma\to \R^3$ is a smooth embedding. Denoting by $g_\phi$ the Riemannian metric on $\Sigma_\gamma$ defined as the pull back by $\phi$ of the Euclidean metric of $\R^3$, one clearly has
 $$\nu_k(S)=\nu_k (\Sigma_{\gamma},g_\phi) \qquad \mbox{ and } \qquad |S|=|(\Sigma_{\gamma},g_\phi)|.$$
 This immediately shows that $ \nu^*_k (\gamma) \le\sup_{g\in\mathcal R (\Sigma_\gamma) } \nu_k (\Sigma_{\gamma},g)\vert(\Sigma_{\gamma},g)\vert.$
 
 Conversely, given any Riemannian metric $g\in\mathcal R (\Sigma_\gamma)$, it is  well  known that there exists a $C^1$-isometric embedding $\phi$ from $(\Sigma_{\gamma},g)$ into $\R^3$ (see \cite{Ku}). Using standard density results, there exists a sequence $\phi_n:\Sigma_\gamma\to \R^3$ of smooth embeddings that converges to $\phi$ with respect to the $C^1$-topology. 
The metrics $g_n=g_{\phi_n } $ induced by $\phi_n$ are  quasi-isometric to $g$ and the corresponding  sequence of quasi-isometry ratios converges to 1.   Therefore, using  \eqref{quasi-eigenvalue} and \eqref{quasi-volume}, $\lim_n \nu_k (\Sigma_{\gamma},g_n)=  \nu_k (\Sigma_{\gamma},g) $ and $\lim_n \vert(\Sigma_{\gamma},g_n)\vert =\vert(\Sigma_{\gamma},g)\vert$.  Hence,
the sequence of surfaces $S_n=\phi_n (\Sigma_\gamma) \in \mathcal M (\gamma)$  satisfies 
 $$\lim_n \nu_k (S_n) \vert S_n\vert=\lim_n \nu_k (\Sigma_{\gamma},g_n)\vert(\Sigma_{\gamma},g_n)\vert  =  \nu_k (\Sigma_{\gamma},g) \vert(\Sigma_{\gamma},g)\vert .$$
This completes the proof of the Lemma.  
\end{proof}

%According to Korevaar \cite{}, if $\Sigma_{\gamma}$ is a compact surface, then  the supremum 
%$$\sup_g \nu_k (g) |(\Sigma_{\gamma},g)|=\sup\left\{\nu_k (g)\ ; \  |(\Sigma_{\gamma},g)|=1 \right\}$$ is bounded above in terms of the genus of $\Sigma_{\gamma}$. To any couple of integers $(k,\gamma)\in \N^2$ we associate the number
%$$\nu_k^*(\gamma)=\sup\{\nu_k(S)\vert S\vert\ ; \ genus(S)=\gamma  \},$$
%where the supremum is taken over compact orientable surfaces of genus $\gamma$. Equivalently,
%$$\nu_k^*(\gamma)=\sup\{\nu_k(S)\ ; \ genus(S)=\gamma \mbox{ and }\ \vert S\vert=1 \},$$
% where $\vert S\vert$ denotes the area of $S$. In the same manner, we define the number $\nu_{*,k}(\gamma)$  as the supremum of $\nu_k(S)\vert S\vert\ $ over compact non-orientable surfaces of genus $\gamma$.
 
 It is known that $\nu_1^*(0)=\nu_1(\mathbb{S}^2, g_{s})=8\pi$, where $ g_{s}$ is  the standard metric of the sphere (see \cite {H}),   $\nu_1^*(1)=\nu_1(\mathbb{T}^2, g_{hex})={8\pi^2\over {\sqrt{3}}}$ , where $g_{hex}$ is the flat metric on the torus associated with the hexagonal lattice (see \cite {N}), and $\nu_2^*(0)=16\pi$ (see \cite {N1}). 
 Moreover, one has the following inequality (see \cite{LY2, EI1})
$$ \nu_1^*(\gamma)  \leq 8 \pi \left\lfloor \frac{\gamma + 3}{2}\right\rfloor ,
$$
where $\lfloor \cdot \rfloor$ denotes the floor function.  Recently, A. Hassannezhad \cite{Asma} proved that there exist universal constants $A >0$ and $B>0$ such that, $\forall (k,\gamma)\in \N^2$,  
  $$\nu_k^*(\gamma)\le A \gamma+ B k.$$
  
 On the other hand,  $\nu_k^*(\gamma)$ admits also a lower bound  in terms of a linear function of $\gamma$ and $k$ as shown in our previous work \cite{CE} where we have also proved that $\nu_k^*(\gamma)$ is {\bf nondecreasing} with respect to $\gamma$.
 
  \begin{thm}\label{surfaces}
 Let $\gamma\ge 0$ and $k\ge 1$ be two integers and let $\gamma_1\dots,\gamma_p \in \N$ and $i_1, \dots,i_p\in \N^*$ be such that $\gamma_1+\dots+\gamma_p=\gamma $ and $i_1+ \dots+i_p=k$. Then 
 \begin{equation}\label{orientsurf}
 \nu_k^*(\gamma)\ge\nu_{i_1}^*(\gamma_1)+\dots +\nu_{i_p}^*(\gamma_p).
 \end{equation}
If the equality holds in \eqref{orientsurf}, then, for every $\varepsilon >0$, there exist $p$ compact orientable surfaces  $S_{1}, \cdots, S_{p}$ of genus $\gamma_1\dots,\gamma_p$, respectively, such that

\smallskip

\begin{itemize}
\item [i)] ${\nu}_{k} (S_{1}\sqcup \cdots\sqcup S_{p}) \ge (1-\varepsilon){\nu}_{k}^*(\gamma) $ ;
\item [ii)] $\forall j\le p$, $(1-\varepsilon){\nu}_{i_j}^*(\gamma_j)\le{\nu}_{i_j}( S_{j}) |S_{j}| \le {\nu}_{i_j}^*(\gamma_j)$ ;
	\item [iii)] $|S_{1}| +\cdots +|S_{p}|=1$ and, $\forall j\le p$, $\frac{{\nu}_{i_j}^*(\gamma_j)}{(1+\varepsilon){\nu}_{k}^*(\gamma)}\le |S_{j}|\le \frac{(1+\varepsilon){\nu}_{i_j}^*(\gamma_j)}{{\nu}_{k}^*(\gamma)}.$ 	
	\end{itemize}
 \end{thm}
  Before giving the proof of this theorem we recall that  
if $S_1$ and $S_2$ are two closed orientable surfaces in $\R^3$, then the spectrum $\left\{\nu_k(S_1\sqcup S_2 )\right\}_{k\ge0}$ of their  disjoint union is given by 
 the re-ordered union of the spectra of $S_1$ and $S_2$ (in particular, $\nu_0(S_1\sqcup S_2 ) =\nu_1(S_1\sqcup S_2 ) =0$). The following lemma shows that this spectrum of $S_1\sqcup S_2 $ can be approximated, with arbitrary accuracy, by the spectrum of a closed connected orientable surface of genus $\gamma= \mbox{genus}(S_1)+\mbox{genus}(S_2)$.

  \begin{lemme}\label{disjointunion}
 Let $S_1$ and $S_2$ be two closed surfaces in $\R^3$ of genus $\gamma_1$ and  $\gamma_2$, respectively. There exists a 1-parameter family $S_\delta\in \mathcal M (\gamma)$ of closed  surfaces  of genus $\gamma=\gamma_1+\gamma_2$ such that, for every $k\ge 0$, 
 $$\lim_{\delta \to 0}\nu_k(S_\delta )=\nu_k(S_1\sqcup S_2 )$$
 and
  $$\lim_{\delta \to 0}\vert S_\delta \vert=\vert S_1\sqcup S_2 \vert.$$
 
 \end{lemme}
 
 In particular, the definition of $\nu_k^*(\gamma) $ does not change if we include in $\mathcal M (\gamma)$  the disjoint unions of surfaces $S_1\sqcup \cdots\sqcup S_p $ with $\mbox{genus}(S_1)+\cdots+ \mbox{genus}(S_p )=\gamma$. 
  
  \begin{proof}[Proof of Lemma \ref{disjointunion}]
  We denote by $g_1$ and $g_2$ the Riemannian metrics induced on $S_1$ and $S_2$, respectively. In what follows, we will show how to construct a 1-parameter family $g_\delta$ of Riemannian metrics on the connected sum $S$ of  $S_1$ and $S_2$ so that 
  $\lim_{\delta \to 0}\nu_k(S,g_\delta )=\nu_k(S_1\sqcup S_2 )$
 and
$\lim_{\delta \to 0}\vert (S,g_\delta) \vert=\vert S_1\sqcup S_2 \vert.$
  Using arguments as in the proof of Lemma \ref{C1-isometric}, we easily see that this family of Riemannian surfaces $(S,g_\delta )$ gives rise to a family of embedded surfaces $S_\delta \in \mathcal M (\gamma_1+\gamma_2)$ which satisfies the conditions of the statement.  For the sake of clarity we divide the proof into several steps.

\underline{Step 1} : 
Let  $x_1\in S_1$ and  $ x_2\in S_2$ be two arbitrary points. For any sufficiently small  $\delta >0$, Lemma 2.3 of \cite{CE} tells us that the metrics $g_1$ and $g_2$ of $S_1$ and $S_2$ are $(1+\delta)$-quasi-isometric to other metrics $g_{1,\delta}$ and $g_{2,\delta}$   which are Euclidean around $x_1$ and $x_2$. As in the proof of Lemma \ref{C1-isometric}, we use   \eqref{quasi-eigenvalue}  to deduce that $\lim_{\delta\to 0} \nu_k(S_i,g_{i,\delta}) =\nu_k(S_i,g_i)$ and, consequently, 
\begin{equation}\label{step1}
\lim_{\delta\to 0} \nu_k\left( (S_1,g_{1,\delta}) \sqcup(S_2,g_{2,\delta}) \right)=\nu_k\left( (S_1,g_{1}) \sqcup(S_2,g_{2}) \right).
\end{equation}

\medskip
\underline{Step 2} : Let $(S,g)$ be a Riemannian surface which is flat around a point $x\in S$. For every sufficiently small $\varepsilon >0$,   the metric $g$ can be deformed in the complement of the geodesic ball of radius $\varepsilon$ into  a metric $g_{\varepsilon}$ which is $(1+2\varepsilon)$-quasi-isometric to $g$ and so that the geodesic annulus $\mathcal A(x,\varepsilon,\varepsilon+ \varepsilon^2)$ centered at $x$ with inner and outer radii   $\varepsilon$ and $\varepsilon+ \varepsilon^2$,  is isometric to the cylinder $S^1_{\varepsilon}\times(\varepsilon,\varepsilon+\varepsilon^2)$, where $S^1_{\varepsilon}$ is the circle of radius $\varepsilon$.

\medskip
Indeed, let us choose $\varepsilon$ so that $g$ is flat in the geodesic ball $B(x,2\varepsilon)$ of radius $2\varepsilon$ centered at $x$ that we identify  with the Euclidean ball $B(O,2\varepsilon)\subset \mathbb R^2$.  Using polar coordinates, we may write 
$$g = dr^2+r^2 d\theta^2$$
with  $r\le 2\varepsilon$ and $\theta \in [0,2\pi]$.
We consider the family $g_{\varepsilon}$ of metrics on $S$ which coincide with $g$ in the complement of the  annulus $\mathcal A(x,\varepsilon,2\varepsilon)$ and whose restriction to  this  annulus (identified with $\mathcal A(0,\varepsilon,2\varepsilon)\subset \mathbb R^2$)   is given by
$$
g_{\varepsilon}(r,\theta)= dr^2+\psi_{\varepsilon}^2(r) d\theta^2,
$$
with $\psi_{\varepsilon}(r)=\varepsilon$ if $\varepsilon \le r\le \varepsilon +\varepsilon^2$,
$\psi(r)=r$ if $\varepsilon+2\varepsilon^2 \le r\le 2\varepsilon$, and $\varepsilon\le \psi_{\varepsilon}( r ) \le \varepsilon +\varepsilon^2$ if  $r\in (\varepsilon +\varepsilon^2 , \varepsilon +2\varepsilon^2 )$. Notice that we do not need to define $\psi_{\varepsilon}$ more explicitly since only $\psi_{\varepsilon}$  will be used and not its derivatives.

\medskip
On the annulus $\mathcal A(0,\varepsilon,\varepsilon+\varepsilon^2)$ the metric $g_{\varepsilon}$ coincides with the cylindrical metric $dr^2+{\varepsilon}^2 d\theta^2$, that is $\mathcal A(x,\varepsilon,\varepsilon+ \varepsilon^2)$  is isometric to $S^1_{\varepsilon}\times(\varepsilon,\varepsilon+\varepsilon^2)$. On the other hand, the metric $g_{\varepsilon}$ is clearly quasi-isometric to the Euclidean metric $g = dr^2+r^2 d\theta^2$ on $\mathcal A(0,\varepsilon,2\varepsilon)$ with 
$$\min\left(1,\frac{\psi_{\varepsilon}^2(r)}{r^2} \right) g \le g _{\varepsilon}\le \max\left(1,\frac{\psi_{\varepsilon}^2(r)}{r^2} \right) g .$$
From the definition of  $\psi_{\varepsilon}$  one has, $\forall r\in (\varepsilon, 2\varepsilon )$,
%$$\frac{\varepsilon^2}{(\varepsilon+2\varepsilon^2)^2} \le \frac{\psi_{\varepsilon}^2(r)}{r^2} \le \frac{(\varepsilon+2\varepsilon^2)^2}{\varepsilon^2},$$ that is
$$
\frac{1}{(1+2\varepsilon)^2} \le \frac{\psi_{\varepsilon}^2(r)}{r^2} \le (1+2\varepsilon)^2.
$$
Since $g_{\varepsilon}$  coincides with $g$ in the complement of $\mathcal A(x,\varepsilon,2\varepsilon)$,   the metric $g_{\varepsilon}$ is in fact globally $(1+2\varepsilon)$-quasi-isometric to $g$.

\medskip
\underline{Step 3} : Construction of the family of metrics $g_\delta$. 

\noindent Given a sufficiently small $\delta >0$, we first apply Step 1 and replace the metrics $g_1$ and $g_2$  by $g_{1,\delta}$ and $g_{2,\delta}$   so that, for each $i=1,2$,  $(S_i,g_{i,\delta}) $ is flat around a point $x_i\in S_i$. Thanks to Step 2, for every positive $\varepsilon <\varepsilon_0(\delta)$,  we define on $S_i$  a metric  $g_{i,\delta,\varepsilon}$ which is $(1+2\varepsilon)$-quasi-isometric to $g$ and so that the geodesic annulus $\mathcal A(x_i,\varepsilon,\varepsilon+ \varepsilon^2)$  is isometric to the cylinder $S^1_{\varepsilon}\times(\varepsilon,\varepsilon+\varepsilon^2)$.  Thus, one can smoothly  glue $(S_1\setminus B(x_1,\varepsilon), g_{1,\delta,\varepsilon})$ and $(S_2\setminus B(x_2,\varepsilon),g_{2,\delta,\varepsilon})$ along their boundaries and obtain a smooth Riemannian surface $(S,g_{\delta,\varepsilon})$  of genus  $\gamma=\gamma_1+\gamma_2$.

Let us denote by $\lambda_k(\delta,\varepsilon) $ (resp.  $\mu_k(\delta,\varepsilon) $) the eigenvalues of the disjoint union of $(S_1\setminus B(x_1,\varepsilon), g_{1,\delta})$ and $(S_2\setminus B(x_2,\varepsilon),g_{2,\delta})$ with Dirichlet (resp. Neumann)  boundary conditions. Similarily, we denote by $\bar\lambda_k(\delta,\varepsilon) $ (resp.  $\bar\mu_k(\delta,\varepsilon) $) the eigenvalues of the disjoint union of $(S_1\setminus B(x_1,\varepsilon), g_{1,\delta,\varepsilon})$ and $(S_2\setminus B(x_2,\varepsilon),g_{2,\delta,\varepsilon})$ with Dirichlet (resp. Neumann)  boundary conditions.
From the min-max principle we have  the following inequalities: 
$$ \bar\mu_k(\delta,\varepsilon) \le \nu_k(S,g_{\delta,\varepsilon}) \le \bar\lambda_k(\delta,\varepsilon).$$
Moreover, since $ g_{i,\delta,\varepsilon}$  is $ (1+2\varepsilon)$-quasi-isometric to $ g_{i,\delta}$, one has using   \eqref{quasi-eigenvalue}, 
$$ (1+2\varepsilon)^{-6}\mu_k(\delta,\varepsilon) \le \bar\mu_k(\delta,\varepsilon)   \quad \mbox{and}\quad \bar\lambda_k(\delta,\varepsilon) \le (1+2\varepsilon)^{6}\lambda_k(\delta,\varepsilon).$$
Therefore,
$$(1+2\varepsilon)^{-6}\mu_k(\delta,\varepsilon)\le \nu_k(S,g_{\delta,\varepsilon}) \le(1+2\varepsilon)^{6}\lambda_k(\delta,\varepsilon).$$
On the other hand, according to \cite{Anne},          $\lambda_k(\delta,\varepsilon) $ (resp.  $\mu_k(\delta,\varepsilon) $) converges as $\varepsilon \to\infty$, to the $k$-th eigenvalue  of the disjoint union of $(S_1, g_{1,\delta})$ and $(S_2,g_{2,\delta})$. Thus, for every $k\ge 0$,
$$\lim_{\varepsilon\to 0}\nu_k(S,g_{\delta,\varepsilon}) =\nu_k \left((S_1, g_{1,\delta}) \sqcup(S_2,g_{2,\delta})\right).$$
In particular, there exists  $\varepsilon (\delta)>0 $ such that, for every $k\le \frac 1 \delta$, 
$$\vert \nu_k(S,g_{\delta,\varepsilon(\delta)}) - \nu_k \left((S_1, g_{1,\delta}) \sqcup(S_2,g_{2,\delta})\right) \vert < \delta.$$
Thus, if we set $g_\delta= g_{\delta,\varepsilon(\delta)}$, then using the last inequality and \eqref{step1}, we will have,  for every $k\ge 0$,
$$\lim_{\delta\to 0}\nu_k(S,g_{\delta}) = \nu_k \left((S_1, g_{1}) \sqcup(S_2,g_{2})\right) .$$
As for  the area, from the construction of $g_\delta$, it is clear  that $\vert (S,g_{\delta})\vert $ tends to
$\vert S_1\vert+\vert S_2\vert$ as $\delta\to 0$.

\end{proof}

\begin{proof}[Proof of Theorem \ref{surfaces}]
Let $\varepsilon$ be any positive real number and let $S_{1}, \cdots ,S_{p}$ be a family of compact orientable surfaces such that, for each positive $j\le p$, $\mbox{genus}(S_j) =\gamma_j$ and $${\nu}_{i_j}( S_{j}) |S_{j}|> {\nu}_{i_j}^*(\gamma_j)- \varepsilon.$$
After rescaling, we may assume that 

$${\nu}_{i_j}( S_{j})  = {\nu}_{k}^*(\gamma)\qquad \mbox{and} \qquad |S_{j}| >\frac{{\nu}_{i_j}^*(\gamma_j)-\varepsilon}{{\nu}_{k}^*(\gamma)}.$$ 
One has, using arguments as in the proof of Theorem \ref{thm dirandneu},
\begin{eqnarray*}
\#\left\{l\in \N\; ; \; \nu_l(S_{1}\sqcup \cdots\sqcup S_{p})<{\nu}_{k}^*(\gamma)\right\} &=&\sum_{j=1}^p \#\left\{l\in \N\; ; \; \nu_l(S_j)<{\nu}_{k}^*(\gamma)\right\}\\
&\le & \sum_{j=1}^p i_j =k.
\end{eqnarray*}
Consequently, 
 $$\nu_k (S_{1}\sqcup \cdots\sqcup S_{p})\ge{\nu}_{k}^*(\gamma).$$
From Lemma  \ref{disjointunion} and the definition of ${\nu}_{k}^*(\gamma)$, one then deduces the following: 
$$\vert S_{1}\sqcup \cdots\sqcup S_{p}\vert = \vert S_{1}\vert+ \cdots+\vert S_{p}\vert \le 1.$$
This leads to 
$$ \sum_{j=1}^p \frac{{\nu}_{i_j}^*(\gamma_j)-\varepsilon}{{\nu}_{k}^*(\gamma)} \le 1,$$
that is,
$$\sum_{j=1}^p \nu_{i_j}^*(\gamma_j)\le  {\nu}_{k}^*(\gamma) +p\varepsilon.$$
This proves the inequality \eqref{orientsurf} since $\varepsilon $ can be chosen arbitrarily small.

Assume that the equality holds in   \eqref{orientsurf} . We can follow the same arguments as in the proof of Theorem \ref{thm dirandneu} and conclude.
\end{proof}

\begin{rem}
A direct consequence of Theorem \ref{surfaces} is that, for every $\gamma \ge 0$ and every $k\ge 1$, one has
$${\nu}_{k}^*(\gamma)\ge \sup_{i \le k-1} \left({\nu}_{i}^*(\gamma)+{\nu}_{k-i}^*(0)\right).$$
In particular, ${\nu}_{k}^*(\gamma)\ge {\nu}_{k-1}^*(\gamma)+8\pi.$
Therefore, Theorem \ref{surfaces}  improves our previous results (Theorem C and Corollary 4 of \cite{CE}).
\end{rem}
The following theorem deals with the situation where ${\nu}_k^*(\gamma)$ is approached by the $k$-th eigenvalue of a disjoint union of two surfaces. 
\begin{thm}\label{WKSURF} Let $\gamma\ge 0$ and $k\ge 2$ be two integers and assume that there exist two compact orientable surfaces $S_1$ and $S_2$  of genus $\gamma_1$, $\gamma_2$, respectively, such that
$\vert S_1\vert+\vert S_2\vert =1$, $\gamma_1+\gamma_2=\gamma$, and
\begin{equation}\label{WKsurf}
{\nu}_k(S_1 \sqcup S_2) \ge {\nu}_k^*(\gamma) -\varepsilon
\end{equation}
for some $\varepsilon \ge 0 $.
Then there exists an integer $ i \in\{1,\cdots, k-1\} $ such that 
$$0\le {\nu}_k^*(\gamma) -\left\{ {\nu}_i^*(\gamma_1)+{\nu}_{k-i}^* (\gamma_2)\right\}\le \varepsilon ,$$ 
$$ 0\le {\nu}_i^*(\gamma_1)-\nu_i(S_1)|S_1|\le\varepsilon \; \;\; \mbox{and} \;\; \;0\le{\nu}_{k-i}^*(\gamma_2)- {\nu}_{k-i} (S_2)|S_2|\le 
\varepsilon .$$

%$${\nu}_k^*(\gamma) \ge {\nu}_i^*(\gamma_1)+{\nu}_{k-i}^* (\gamma_2)\ge {\nu}_k^*(\gamma) -\varepsilon ,$$ 
%$$\nu_i(S_1)|S_1|\ge {\nu}_i^*(\gamma_1)-\varepsilon \; \;\; \mbox{and} \;\; \; {\nu}_{k-i} (S_2)|S_2|\ge {\nu}_{k-i}^*(\gamma_2)-\varepsilon .$$

%Moreover, if the area of each of the surfaces $S_1$ and $S_2$  exceeds $\varepsilon/{\l}_k^*(\gamma)$, then the integer $i$ is between $1$ and $k-1$.
\end{thm}

\begin{proof}
Since the spectrum of $S_1\sqcup S_2 $ is the re-ordrered union of the spectra of $S_1$ and $S_2$, the eigenvalue $\nu_k(S_1\sqcup S_2 )$ belongs to this union and, moreover, 
\begin{equation}\label{WKsurf1}
 \#\left\{j\in \N\; ; \; \nu_j(S_1)<\nu_k(S_1\sqcup S_2 )\right\} \; +\;  \#\left\{j\in \N\; ; \; \nu_j(S_2)<\nu_k(S_1\sqcup S_2 )\right\} \le k
 \end{equation}
and
\begin{equation}\label{WKsurf2}
 \#\left\{j\in \N\; ; \; \nu_j(S_1)\le\nu_k(S_1\sqcup S_2 )\right\} \; + \; \#\left\{j\in \N\; ; \; \nu_j(S_2)\le\nu_k(S_1\sqcup S_2 )\right\} \ge k+1 
 \end{equation}
 (recall that the numbering of the eigenvalues start from zero).
 Hence, there exists at least one integer $j\in\{1,\dots, k\}$ such that $\nu_j(S_1)=\nu_k(S_1\sqcup S_2 )$ or $\nu_j(S_2)=\nu_k(S_1\sqcup S_2 )$. Assume that the first alternative occurs and let $i$ be the least positive integer such that $\nu_i(S_1)=\nu_k(S_1\sqcup S_2 )$. 
We necessarily have $\nu_{k-i}(S_2)\ge \nu_k(S_1\sqcup S_2 )$ since, otherwise, the $k+1$ eigenvalues $\nu_0(S_1),\cdots, \nu_{i-1}(S_1)$ and $\nu_0(S_2),\cdots, \nu_{k-i}(S_1)$ would be strictly less than $\nu_k(S_1\sqcup S_2 )$  which contradicts \eqref{WKsurf2}.  Thus, $i\le k-1$ and
\begin{equation}\label{WKsurf3}
 \nu_k(S_1\sqcup S_2 )=\nu_k(S_1\sqcup S_2 )|S_1| +\nu_k(S_1\sqcup S_2 ) |S_2|\le \nu_i(S_1)|S_1|+\nu_{k-i}(S_2)|S_2|. 
\end{equation}
Since $\nu_i(S_1)|S_1|\le {\nu}_i^*(\gamma_1)$ and $\nu_{k-i}(S_2)|S_2|\le {\nu}_{k-i}^*(\gamma_2)$, we get
$$
{\nu}_k^*(\gamma) -\varepsilon\le \nu_k(S_1\sqcup S_2 )\le {\nu}_i^*(\gamma_1)+{\nu}_{k-i}^*(\gamma_2). 
$$

Now, $\nu_i(S_1)|S_1|\ge {\nu}_i^*(\gamma_1)-\varepsilon$. Otherwise, we would have, thanks to \eqref{WKsurf3} and Theorem  \ref{surfaces},
$$\nu_k(S_1\sqcup S_2 )\le \nu_i(S_1)|S_1|+\nu_{k-i}(S_2)|S_2|< {\nu}_i^*(\gamma_1)-\varepsilon + {\nu}_{k-i}^*(\gamma_2)\le {\nu}_k^*(\gamma)-\varepsilon $$
which contradicts the assumption of the theorem. 
The same argument leads to the inequality ${\nu}_{k-i} (S_2)|S_2|\ge {\nu}_{k-i}(n)+ 
\varepsilon $.

%Now, if $i=k$, then $\nu_{k}(S_2)\ge \nu_k(S_1\sqcup S_2 )$.
\end{proof}

As a consequence of Theorem \ref{WKSURF}, we obtain the following Wolf-Keller type result.
\begin{cor}\label{WKSURF1} Let $\gamma\ge 0$ and $k\ge 2$ be two integers and assume that there exist two compact orientable surfaces $S_1$ and $S_2$  of genus $\gamma_1$, $\gamma_2$, respectively, such that
$\vert S_1\vert+\vert S_2\vert =1$, $\gamma_1+\gamma_2=\gamma$, and
\begin{equation}\label{WKsurfaces}
{\nu}_k(S_1 \sqcup S_2) = {\nu}_k^*(\gamma). 
\end{equation}
Then there exists an integer $ i \in\{1,\cdots, k-1\} $ such that 
$${\nu}_k^*(\gamma) = {\nu}_i^*(\gamma_1)+{\nu}_{k-i}^* (\gamma_2)= \sup_{j=1,\cdots, k-1} \left\{{\nu}_j^*(\gamma_1)+{\nu}_{k-j}^* (\gamma_2)\right\}$$ 
$$\nu_i(S_1)|S_1|= {\nu}_i^*(\gamma_1)  \; \;\; \mbox{and} \;\; \; {\nu}_{k-i} (S_2)|S_2|= {\nu}_{k-i}^*(\gamma_2).$$

\end{cor}

\medskip

\subsection*{Extremal eigenvalues of nonorientable surfaces }
$ $\\
In the non-orientable case, we can similarly define, for every $\gamma\in\N$ and every $k\in\N$, the number $\nu_{*,k}(\gamma)$  as the supremum of $\nu_k(S)\vert S\vert\ $ over compact non-orientable surfaces of genus $\gamma$.

We have  $\nu_{*,1}(1)=\nu_1(\mathbb{R}P^2, g_{s})=12\pi$ where $ g_{s}$ is the standard metric of the projective plane (see \cite{LY2}), and $\nu_{*,1}(2)=\nu_1 (\mathbb{K}^2,g_0)=12 \pi E(2\sqrt 2/3)\simeq13.365\,\pi ,$ 
where $g_0$ is a non flat  metric of revolution  on the Klein bottle and $E(2\sqrt 2/3)$ is the complete elliptic integral of the second kind evaluated at $\frac{2\sqrt2}3$ (see \cite{EGJ}). Moreover, one has the following inequalities (see \cite{LY2, EI1})
$$ \nu_{*,1}(\gamma)  \leq 24 \pi \left\lfloor \frac{\gamma + 3}{2}\right\rfloor ,
$$
where $\lfloor \cdot \rfloor$ denotes the floor function.
The same reasoning as in the orientable case leads to the following results :

  \begin{thm}\label{nonorientsurfaces}
 Let $\gamma\ge 0$ and $k\ge 1$ be two integers and let $\gamma_1\dots,\gamma_p$ and $i_1, \dots,i_p$ be such that $\gamma_1+\dots+\gamma_p=\gamma $ and $i_1+ \dots+i_p=k$. Then 
 \begin{equation}\label{nonorientsurf}
 \nu_{*,k}(\gamma)\ge\nu_{*,i_1}(\gamma_1)+\dots +\nu_{*,i_p}(\gamma_p).
  \end{equation}
  If the equality holds in \eqref{orientsurf}, then, for every $\varepsilon >0$, there exist $p$ compact orientable surfaces  $S_1, \cdots, S_p$ of genus $\gamma_1\dots,\gamma_p$, respectively, such that

\smallskip

\begin{itemize}
\item [i)] ${\nu}_{k} (S_1\sqcup \cdots\sqcup S_p) \le (1+\varepsilon){\nu}_{*,k}(\gamma) $ ;
	\item [ii)] $|S_1| +\cdots +|S_p|=1$ and, $\forall j\le p$, $\frac{{\nu}_{*, i_j}(\gamma_j)}{(1+\varepsilon){\nu}_{*,k}(\gamma)}\le |S_j|\le \frac{(1+\varepsilon){\nu}_{*,i_j}(\gamma_j)}{{\nu}_{*,k}(\gamma)}$; 	
	\item [iii)] $\forall j\le p$, ${\nu}_{*,i_j}(\gamma_j)\le{\nu}_{i_j}( S_j) |S_j|^{2/n} \le (1+\varepsilon){\nu}_{*,i_j}(\gamma_j)$.
\end{itemize}

 \end{thm}

\bibliographystyle{plain}
\bibliography{biblio1}

\end{document}